\documentclass[12pt, reqno]{amsart}
\usepackage{mathrsfs}
\usepackage{graphicx}
\usepackage{amsmath}
\usepackage{amsfonts}
\usepackage{amssymb}
\newtheorem{Exam.}{Exam.}
\newtheorem{definition}{Definition}

\newtheorem{theorem}{Theorem}
\newtheorem{lem}[theorem]{Lemma}
\newtheorem{proposition}[theorem]{Proposition}

\newtheorem{conjecture}[theorem]{Conjecture}

\usepackage{amsthm,amsmath,amssymb,url,cite,color}

\def\red{\textcolor{red}}

\def\L{\mathcal L}

\numberwithin{equation}{section}

\def\Hom{\mathrm{Hom}}

\def\Epi{\mathrm{Epi}}
\def\End{\mathrm{End}}
\def\N{\mathbb N}
\def\Z{\mathbb Z}
\def\Q{\mathbb Q}
\def\S{\mathfrak{S}}
\begin{document}

\title{On the number of congruence classes of paths}

\author{Zhicong Lin}
\address[Zhicong Lin]{Department of Mathematics and Statistics, Lanzhou University, China, and  Universit\'{e} de Lyon; Universit\'{e} Lyon 1; Institut Camille Jordan; UMR 5208 du CNRS; 43, boulevard du 11 novembre 1918, F-69622 Villeurbanne Cedex, France}
\email{lin@math.univ-lyon1.fr}

\author{Jiang Zeng}
\address[Jiang Zeng]{Universit\'{e} de Lyon; Universit\'{e} Lyon 1; Institut Camille Jordan; UMR 5208 du CNRS; 43, boulevard du 11 novembre 1918, F-69622 Villeurbanne Cedex, France}
\email{zeng@math.univ-lyon1.fr}

\date{\today}
\maketitle

\begin{abstract}
Let $P_n$  denote the undirected path of  length $n-1$. 
The cardinality of the set of congruence classes induced by the graph homomorphisms  
from $P_n$ onto $P_k$  is determined. This settles an open problem of 
 Michels and Knauer~(Disc. Math., 
309\ (2009)\  5352-5359). Our result is based on a new
proven formula of the number of homomorphisms between paths.

\end{abstract}

\section{Introduction}

We use standard notations and terminology of graph theory in \cite{bm} or \cite[Appendix]{st}. 
The graphs considered here are finite and undirected without multiple edges and loops. Given a graph $G$, we write $V(G)$ for the vertex set and $E(G)$ for the edge set. 
A {\it homomorphism} from a graph $G$ to a graph $H$ is a mapping $f:V(G)\to V(H)$ 
such that the images of adjacent vertices are adjacent.
An \emph{endomorphism} of a graph is a homomorphism from the graph to itself. Denote by $\Hom(G, H)$ the set of homomorphisms from $G$ to $H$ and by $\End(G)$ the set of endomorphisms of a graph $G$. 
For any finite set $X$ we denote by $|X|$ the cardinality of $X$.
 A {\it path} with $n$ vertices is a graph whose vertices can be labeled $v_1,...,v_n$ so that $v_i$ and $v_j$ are adjacent if and only if $|i-j|=1$; let $P_n$ denote such a graph with $v_i=i$ for $1\le i\le n$. Every endomorphism $f$ on $G$ induces a partition $\rho$ of $V(G)$, 
 also called \emph{the congruence classes induced by $f$}, 
 with vertices in the same block if they have the same image. 
 
Let $\mathscr{C}(P_{n})$ denote the set of endomorphism-induced partitions of $V(P_{n})$,
 and let $|\rho|$ denote
  the number of blocks in a partition $\rho$.  
  For example, if $f \in \End(P_{4})$ is defined by $f(1)=3$, $f(2)=2$, $f(3)=1$, $f(4)=2$, then the induced partition 
  $\rho$ is $\{\{1\},\{2,4\},\{3\}\}$ and $|\rho|=3$.

The problem  of counting the homomorphisms from $G$ to $H$ is difficult in general. However, some algorithms and formulas for 
computing the number of homomorphisms of paths have been published recently (see \cite{ar, aw, mk}). In particular, Michels and 
Knauer~\cite{mk} give  an algorithm based on the \emph{epispectrum} $\Epi(P_{n})$ of a path $P_n$.  They define $\Epi(P_{n}) = (l_{1}(n), . . . , l_{n-1}(n))$, where
\begin{equation}\label{eq:b}
l_{k}(n) = | \{  \rho \in \mathscr{C}(P_{n}) \  : \  |\rho| = n - k + 1 \} |.
\end{equation} 
Here a misprint in the definition of $l_{k}(n) $ in  \cite{mk} is corrected. 

In \cite{mk}, based on the first values of $l_k(n)$, Michels and Knauer speculated  the following conjecture.
\begin{conjecture}
There exists a polynomial $f_{k} \in \Q[x]$ 
 with $deg(f_{k}) = \lceil (k - 2)/2 \rceil$ such that for a fixed $n_{k}$ (most probably $n_k=2k$) the equality 
 $l_{k}(n) = f_{k}(n)$ holds for $n \geq n_{k}$. 
 \end{conjecture}
 
The aim  of this paper is to confirm this conjecture by giving
an explicit formula for the polynomial $f_k$. 
For this purpose, we shall  prove a  new formula 
for the number of homomorphisms from 
$P_n$ to $P_k$, which is the content of the following theorem.

 \begin{theorem} \label{th:1}
For any positive integers $n$ and $k$,
\begin{align} 
|\Hom(P_{n}, P_{k})|&=k\times2^{n-1}\ \label{eq:theorem2}\\
&-\sum\limits_{i=0}^{n-2}
2^{n-1-i}\sum_{j\in \mathbb{Z}}\left({i \choose \lceil\frac{i}{2}\rceil-j(k+1)}-{i \choose
\lfloor\frac{i+k+1}{2}\rfloor-j(k+1)}\right).\nonumber
\end{align}
\end{theorem}

From the above theorem we are able to derive 
the following  main result. 

\begin{theorem} \label{th:2}
If $n\geq2k$, then 
\begin{align}\label{eq:main}
l_{k}(n)=\binom{n-1}{\left\lceil\frac{k}2\right\rceil-1} + \binom{n-1}{\lfloor\frac{k}2\rfloor-1}.
\end{align}
Equivalently, the above formula can be rephrased as follows 
\begin{align}
l_{2k}(n)=2{n-1 \choose k-1}, \qquad
l_{2k+1}(n)={n \choose k}.
\end{align}
\end{theorem}

When $n\geq 2k$, Theorem~3 shows immediately that $l_{k}(n)$
 is a polynomial in $n$ of degree $ \lceil (k - 2)/2 \rceil$. This proves Conjecture~1.
In particular, we have $l_1(n)=1,
l_2(n)=2, l_3(n)=n, l_4(n)=2(n-1),
l_5(n)=\frac{1}{2}n(n-1)$ and $l_6(n)=(n-1)(n-2)$, which coincide with the conjectured values  in \cite{mk} after shifting the index by 1.

In the next section, we  first recall  some basic counting 
 results about  the  lattice paths
and then prove Theorem~2. In Section~3 we give the proof of Theorem~3.
\section{The number of homomorphisms between  paths}
One can enumerate homomorphisms  from $P_n$ to $P_k$ by picking a fixed point as image of 
$1$ and moving to vertices which are adjacent to this vertex, as 
$$f\in \Hom(P_n,P_k)\Leftrightarrow 
\forall x\in \{1, \ldots, n-1\}: \{f(x),  f(x+1)\}\in E(P_k).
$$
 Hence, one  can describe all possible moves through the edge structure of the two  paths. 

For  $1\leq j \leq k$, let
\begin{align}
\Hom^j(P_{n},P_{k})=\{f \in \Hom(P_{n},P_{k}) \ : \ f(1)=j\}.
\end{align}
Obviously, we have
\begin{align}
|\Hom^j(P_{n},P_{k})|=|\{f \in \Hom(P_{n},P_{k}) \ :\ f(n)=j\}|.
\end{align}

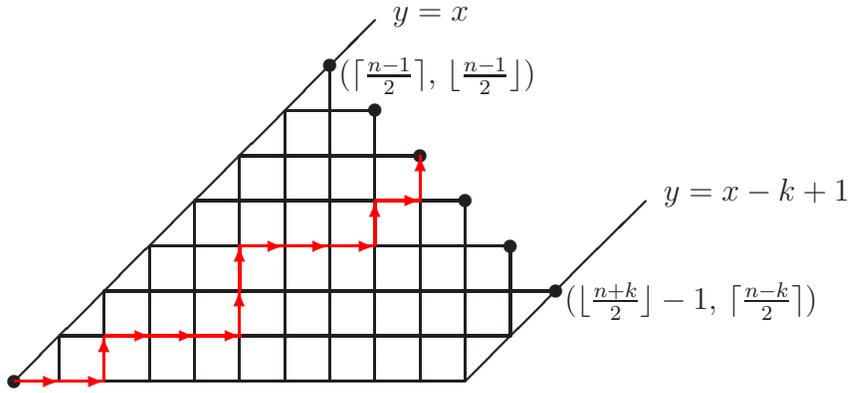
\begin {figure}
\setlength {\unitlength} {1mm}
\begin {picture} (100,55) \setlength {\unitlength} {1.2mm}
\thicklines
\put(10,0){\line(1,0){40}} \put(0,0){\line(1,1){40}} \put(42,40){$y=x$} \put(50,0){\line(1,1){20}}\put(72,20){$y=x-k+1$}
\put(0,0){\circle*{1.5}}  
\put(5,0){\line(0,1){5}} \put(5,5){\line(1,0){5}} \put(25,5){\line(1,0){30}}
\put(10,5){\line(0,1){5}} \put(10,10){\line(1,0){50}}
\put(15,0){\line(0,1){15}} \put(15,15){\line(1,0){10}} \put(40,15){\line(1,0){15}}
\put(20,0){\line(0,1){20}} \put(20,20){\line(1,0){20}}\put(45,20){\line(1,0){5}}
\put(25,0){\line(0,1){5}} \put(25,15){\line(0,1){10}}\put(25,25){\line(1,0){20}}
\put(30,0){\line(0,1){30}} \put(35,0){\line(0,1){35}} \put(40,0){\line(0,1){15}}\put(40,20){\line(0,1){10}}
\put(45,0){\line(0,1){20}} \put(50,0){\line(0,1){20}}
\put(55,5){\line(0,1){10}}\put(30,30){\line(1,0){10}}
\put(35,35){\circle*{1.5}}\put(40,30){\circle*{1.5}}
\put(45,25){\circle*{1.5}}\put(50,20){\circle*{1.5}}
\put(55,15){\circle*{1.5}}\put(60,10){\circle*{1.5}}
\put(36,33){($\lceil\frac{n-1}{2}\rceil$, $\lfloor\frac{n-1}{2}\rfloor$)}
\put(61,8){($\lfloor\frac{n+k}{2}\rfloor-1$, $\lceil\frac{n-k}{2}\rceil$)}
\red{\put(0,0){\vector(1,0){5}}} \red{\put(3.85,0){\vector(1,0){5}}}
\red{\put(7.7,0){\vector(0,1){5}}} \red{\put(6.7,5){\vector(1,0){5}}} \red{\put(10.5,5){\vector(1,0){5}}}
\red{\put(14.4,5){\vector(1,0){5}}}
\red{\put(18.1,5){\vector(0,1){5}}}
\red{\put(17,10){\vector(0,1){5}}}
\red{\put(15.7,15){\vector(1,0){5}}}
\red{\put(19.5,15){\vector(1,0){5}}}
\red{\put(23.5,15){\vector(1,0){5}}}
\red{\put(27.4,15){\vector(0,1){5}}}
\red{\put(26.3,20){\vector(1,0){5}}}
\red{\put(30.1,20){\vector(0,1){5}}}
\end{picture}
\caption{A lattice path from $(0,0)$ to $(9,5)$ that  stays between lines $y=x$ and $y=x-k+1$, where $n=15$ and $k=11$.}
\label{fig1}
\end {figure}

\begin{definition}
A \emph{lattice path}  of length $n$ is a sequence $(\gamma_0, \ldots, \gamma_{n})$ of points $\gamma_i$ in the plan $\Z\times\Z$ for all $0\leq i\leq n$ and such that $\gamma_{i+1}-\gamma_i=(1,0)$ (east-step) or $(0,1)$ (north-step) for $1\leq i\leq n-1$. 
\end{definition}

As shown by Arworn\cite{ar},  we can  
encode each  homomorphism $f\in \Hom^1(P_{n}, P_{k})$ by a lattice path 
$\gamma=(\gamma_0, \ldots, \gamma_{n-1})$ in $\N\times \N$
between the lines $y = x$ and $y = x - k+1$  as follows: 
\begin{itemize}
\item $\gamma_0=(0,0)$, and for $j=1,\ldots, n-1$,   
\item $\gamma_{j+1}=\gamma_j+(1,0)$ if 
$f(j)>f(j-1)$,  
\item $\gamma_{j+1}=\gamma_j+(0,1)$ if $f(j)<f(j-1)$. 
\end{itemize}
For example,  
if  the images of successive vertices of $f\in \Hom(P_{15}, P_{11})$ are 
$$
1,\, 2,\, 3,\, 2,\, 3,\, 4,\, 5,\, 4,\, 3,\, 4,\, 5,\, 6,\, 5,\, 6,\, 5;
$$ 
then   the  corresponding lattice path is  given by  Figure~\ref{fig1}. 

\begin{definition}
For nonnegative integers $n,m,t,s$, Let $\L(n,m)$ be the set of all the lattice paths from the origin to $(n,m)$ and 
$\L(n,m;t,s)$ the set  of lattice paths in  $\L(n,m)$  that
stay between the lines $y = x + t$ and $y = x-s$ (being allowed to
touch them), where  $n + t \geq m \geq n - s$.
\end{definition}

\begin{lem} Let $K=\min(\lfloor\frac{n+k}{2}\rfloor, n)$, then
\begin{align}\label{key1}
|\Hom^1(P_{n},P_{k})|
=\sum\limits_{l=\lceil\frac{n-1}{2}\rceil}^{K-1}|\L(l, n-1-l; 0, k-1)|.
\end{align} 
\end{lem}
\begin{proof}
 It follows from the above correspondence that each homomorphism from $P_n$ to $P_k$  is encoded by a lattice path in some $\L(\#E, \#N; 0, k-1)$, where
 $\#E$ is the number of east-steps and $\#N$ the number of north-steps.  The path structures require that 
 $$
 \#E+\#N=n-1,\quad \#E-\#N\leq k-1,\quad \#E-\#N\geq   0.
 $$
 Therefore, we must have $\#E\geq (n-1)/2$,  $\#E\leq n-1$ and $\#E\leq (k+n-2)/2$.
\end{proof}

To evaluate the sum in \eqref{key1}, we need a formula for the cardinality of $\L(n,m;t,s)$.
First of all,  each  lattice path in $\L(n,m)$ can be encoded by a word  of length $n+m$ on the alphabet $\{A, B\}$ with 
$n$ letters $A$  and $m$ letters $B$. So,
the cardinality  of $\L(n,m)$ is given by the binomial coefficient ${n+m \choose n}$. 
Next, each lattice path  in $\L(n,m)$ which passes above the line $y=x+t$ (or reaching the line $y=x+t+1$) can be 
mapped to a lattice path from $(-t-1, t+1)$ to $(n,m)$ by the \emph{reflection} with respect to the line $y=x+t+1$ (see Figure~\ref{fig2}). Hence, there are 
${n+m \choose n+t+1}$ such lattice paths.  Therefore, 
the number of  lattice paths in $\L(n,m)$ which do not pass above the line $y=x+t$ (or not reaching the line $y=x+t+1$), where $m\leq n+t$, is given by 
$$
{n+m \choose n}-{n+m \choose n+t+1}.
$$
By a similar reasoning, 
we can prove 
the following known result (see \cite[Lemma 4A]{na}, for example). For the reader's convenience, we  provide a sketch of the proof.
\begin{lem}   The cardinality of $\L(n,m;t,s)$ is given by
\begin{equation}\label{le:Wa}
|\L(n,m;t,s)|=\sum_{k \in \mathbb{Z}}\left({n+m \choose
n-k(t+s+2)}-{n+m \choose n-k(t+s+2)+t+1}\right),
\end{equation} where ${n \choose k}=0$ if $k>n$ or $k<0$.
\end{lem}

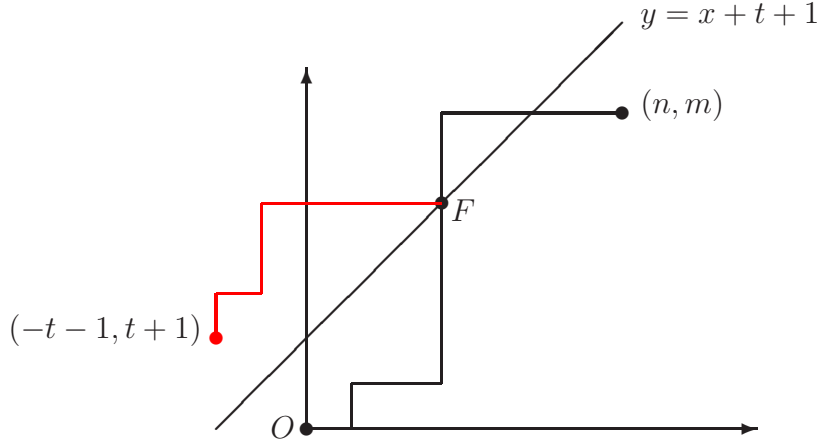
\begin {figure}
\setlength {\unitlength} {1mm}
\begin {picture} (120,60) \setlength {\unitlength} {1.2mm}
\thicklines
\put(30,0){\circle*{1.5}}\put(26,-1){$O$}
\put(30,0){\vector(0,1){40}}\put(30,0){\vector(1,0){50}}
\put(20,0){\line(1,1){45}}\put(67,45){$y=x+t+1$}
\put(35,0){\line(0,1){5}}\put(35,5){\line(1,0){10}}\put(45,5){\line(0,1){30}}
\put(45,35){\line(1,0){20}}
\put(65,35){\circle*{1.5}}\put(67,35){$(n,m)$}
\put(46,23){$F$}\put(45,25){\circle*{1.5}}

\put(-3,10){$(-t-1,t+1)$}

\red{\put(20,10){\circle*{1.5}}
\put(20,10){\line(0,1){5}}\put(20,15){\line(1,0){5}}\put(25,15){\line(0,1){10}}
\put(25,25){\line(1,0){20}}}
\end{picture}
\caption{Reflection of the segment of the path from $O$ to the first reaching  point $F$ with respect to the line $y=x+t+1$.}
\label{fig2}
\end {figure}

\begin{proof}[Sketch of proof]
Let $T$ and $S$ be the  lines $y=x+t+1$ and $y=x-s-1$, respectively.
Let $A_1$ denote the set of lattice paths in $\L(n,m)$  reaching $T$ at least once, regardless of what happens at any other step, and let $A_2$ denote the set of lattice paths in $\L(n,m)$ reaching $T,S$ at least once in the order specified. 
 Generally, let $A_i$ denote the set of lattice paths in $\L(n,m)$ reaching $T,S, \ldots$, alternatively ($i$ times) at least once in the specified order. Let $B_i$ be the set defined in the same way as $A_i$ with $S,T$ interchanged. A standard Inclusive-Exclusive principle argument  yields:
\begin{equation}\label{L}
|\L(n,m;t,s)|={n+m \choose n}+\sum_{i\geq 1}(-1)^i(|A_i|+|B_i|).
\end{equation}
As 
the symmetric point of $(a,b)$ with respect to the line $y=x+c$ is $(b-c, a+c)$,
by repeatedly applying  the reflection principle argument,
we obtain
$$
|A_{2j}|={n+m \choose n+j(t+s+2)}, \quad |A_{2j+1}|={n+m \choose n-j(t+s+2)-(t+1)},
$$
 and 
$$
|B_{2j}|={n+m \choose n-j(t+s+2)}, \quad |B_{2j+1}|={n+m \choose n+j(t+s+2)-(s+1)}.
$$
Substituting this into \eqref{L}  leads to \eqref{le:Wa}.
\end{proof}
\begin{lem}\label{lem4}
For each positive integers $n$ and $k$,
\begin{align}\label{eq4}
|\Hom^1(P_{n},P_{k})|=
\sum_{j\in \mathbb{Z}}
\left({n-1 \choose \left\lceil\frac{n-1}{2}\right\rceil-j(k+1)}-
{n-1 \choose \left\lfloor\frac{n+k}{2}\right\rfloor -j(k+1)}
\right).
\end{align} 
\end{lem}



\begin{proof}
Substituting  \eqref{le:Wa}  into \eqref{key1} and 
exchanging  the order of the summations,
\begin{align}
|\Hom^1(P_{n},P_{k})|
&=\sum_{j\in \mathbb{Z}}\sum\limits_{l=\lceil\frac{n-1}{2}\rceil}^{K-1}
\left({n-1 \choose l-j(k+1)}-{n-1 \choose l+1-j(k+1)}\right)\nonumber\\
&
=\sum_{j\in \mathbb{Z}}\left({n-1 \choose \left\lceil\frac{n-1}{2}\right\rceil-j(k+1)}-{n-1 \choose
K-j(k+1)}\right).\label{eq:last}
\end{align} 
Now, if $n\geq k$,  then $K=\lfloor\frac{n+k}{2}\rfloor$, 
\begin{align}\label{eq:note}
{n-1\choose K-j(k+1)}={n-1 \choose \lfloor\frac{n+k}{2}\rfloor-j(k+1)},
\end{align}
if $k>n$, then $K=n$, since
\begin{equation*}
{n-1 \choose n-j(k+1)}={n-1 \choose \lfloor\frac{n+k}{2}\rfloor-j(k+1)}=0,
\end{equation*}
the equation \eqref{eq:note} is also valid.
 Hence \eqref{eq:last} and  \eqref{eq4} are  equal.
\end{proof}

\begin{proof}[{\bf Proof of  Theorem~1}]
For $f\in \Hom(P_{i+1},P_{k})$ with $i=1,\ldots, n-1$, consider the following  three cases:
\begin{itemize}
\item[(i)]  if $f(i)=1$,  then $f(i+1)=2$ and there are $|\Hom^1(P_{i},P_{k})|$ such homomorphisms.
\item[(ii)]  if $f(i)=k$,  then $f(i+1)=k-1$ and there are $|\Hom^k(P_{i},P_{k})|$ such homomorphisms.
\item[(iii)]  if $f(i)=j$ with $j\in \{2, 3, . . . , k-1\}$, then $f(i+1)=j-1$ or $j+1$ and there are $2|\Hom^j(P_{i},P_{k})|$ such homomorphisms.
\end{itemize}
Summarizing,  we get
$$
|\Hom(P_{i+1},P_{k})|=|\Hom^1(P_{i},P_{k})|+2\sum\limits_{j=2}^{k-1}|\Hom^j(P_{i},P_{k})|+|\Hom^k(P_{i},P_{k})|.
$$
Since  $|\Hom(P_{i},P_{k})|=\sum\limits_{j=1}^{k}|\Hom^j(P_{i},P_{k})|$ and 
 $|\Hom^1(P_{i},P_{k})|=|\Hom^k(P_{i},P_{k})|$,  it follows that 
 $$|\Hom(P_{i+1},P_{k})|=2|\Hom(P_{i},P_{k})|-2|\Hom^1(P_{i},P_{k})|.
 $$
  By iteration, we derive
\begin{align}\label{eq5}
 |\Hom(P_{n},P_{k})|&=2^{n-1}|\Hom(P_{1},P_{k})| - \sum\limits_{i=1}^{n-1}2^{n-i}|\Hom^1(P_{i},P_{k})|\nonumber\\
 &=k \times 2^{n-1} - \sum\limits_{i=1}^{n-1}2^{n-i}|\Hom^1(P_{i},P_{k})|.
 \end{align}
Plugging \eqref{eq4} into  \eqref{eq5}, we obtain \eqref{eq:theorem2} .
\end{proof}

\noindent{\bf Remark.}  The key point in the above proof is  to reduce the counting problem of 
$|\Hom(P_{n}, P_{k})|$ to $|\Hom^1(P_{i},P_{k})|$  for $i=1,\ldots, n-1$. 
Arworn and Wojtylak \cite{aw} give a formula for $|\Hom(P_{n}, P_{k})|=\sum_{j=1}^{k} |\Hom^j(P_{n},P_{k})|$
 without using this reduction. Moreover, their expression for $|\Hom^j(P_{n},P_{k})|$ depends on the parity of $n-j$:
\begin{align}\label{eq:two}
|\Hom^j(P_{n},P_{k})|=\begin{cases}
\sum_{t=-n+1}^{n-1}(-1)^t\sum_{u=0}^{\lfloor\frac{k-1}{2}\rfloor}{n-1 \choose \frac{n-j-1}{2}+u+\lceil\frac{(k+1)t}{2}\rceil}& \text{if $n-j$ is odd}, \\
\sum_{t=-n+1}^{n-1}(-1)^t\sum_{u=0}^{\lceil\frac{k-1}{2}\rceil}{n-1 \choose \lfloor\frac{n-j-1}{2}\rfloor+u+\lfloor\frac{(k+1)t}{2}\rfloor}&
\text{ if $n-j$ is even}. 
 \end{cases}
 \end{align}
 Note that  Lemma~\ref{eq4} unifies the two cases in \eqref{eq:two}  
 when $j=1$.
\medskip

When $k=n$, we can deduce  a simple formula for the number of endomorphisms of $P_n$ (see \texttt{http://oeis.org/A102699}) by applying 
two binomial coefficient identities.
\begin{lem}\label{lem6}
 For $m\geq 1$, the following identities hold
\begin{align}
\sum_{k=0}^{m-1}{2k\choose  k} 2^{2m-1-2k}&=m{2\,m\choose m},\label{eq1}\\
\sum_{k=0}^{m-1}{2k+1\choose  k} 2^{2m-1-2k}&= (m+1){2\,m+1\choose m}-{2}^{2\,m}.\label{eq2}
\end{align}
\end{lem}
\begin{proof}
We prove \eqref{eq1} by induction on $m$. 
Clearly \eqref{eq1} is true for $m=1$. If it is true for $m\geq 1$, then for $m+1$,  
the left-hand side after cutting out the last term,   can be  written as
\begin{align*}
2^2\sum_{k=0}^{m-1}{2k\choose  k} 2^{2m-1-2k}+ 2{2m\choose m}&=4m{2\,m\choose m}+2{2\,m\choose m}\\
&=(m+1){2m+2\choose m+1}.
\end{align*}
Thus \eqref{eq1} is  proved.  Similarly  we can prove \eqref{eq2}.
\end{proof}

\begin{proposition}
For $n\geq 1$, 
\begin{align}
&|\End(P_{n})|
=\begin{cases}
(n+1)2^{n-1}-(2n-1)\binom{n-1}{(n-1)/2}
&\text{ if $n$ is odd,}\\
(n+1)2^{n-1}-n\binom{n}{n/2}
&\text{ if $n$ is even.}\\
\end{cases}
\label{eq:eo}
\end{align}
\end{proposition}
\begin{proof}  
When $k=n$, Theorem~\ref{th:1} becomes
 \begin{equation}
|\End(P_{n})|=n\times2^{n-1}-\sum\limits_{i=0}^{n-2}
2^{n-1-i}\times {i \choose \lceil\frac{i}{2}\rceil}.
\label{eq:eo1}
\end{equation}
By Lemma~\ref{lem6}, if $n$ is even, say $n=2m$,  then
\begin{align*}
\sum\limits_{i=0}^{n-2}
2^{n-1-i}\times {i \choose \lceil\frac{i}{2}\rceil}&=
\sum_{k=0}^{m-2}{2k+1\choose  k} 2^{2m-2-2k}
+\sum_{k=0}^{m-1}{2k\choose  k} 2^{2m-1-2k}\\
&= 2m {2\,m-1\choose m-1}-{2}^{2\,m-1}+m{2\,
m\choose m};
\end{align*}
if $n$ is odd, say $n=2m+1$, then
\begin{align*}
\sum\limits_{i=0}^{n-2}
2^{n-1-i}\times {i \choose \lceil\frac{i}{2}\rceil}&=
\sum_{k=0}^{m-1}{2k+1\choose  k} 2^{2m-1-2k}
+\sum_{k=0}^{m-1}{2k\choose  k} 2^{2m-2k}\\
&= (m+1){2\,m+1\choose m}-{2}^{2\,m}+2m{2\,m\choose m}.
\end{align*}
Substituting these into \eqref{eq:eo1} we obtain the desired result.
\end{proof}

\section{Proof of theorem~\ref{th:2}}
We first  establish   three lemmas.
For any $n\geq 1$, let $[n]=\{1, \ldots, n\}$, which is   $V(P_n)$.
Denote by $\S_n$  the set of
 permutations of $[n]$.
For $1\leq k \leq n$,  denote by $\Epi(P_n, P_k)$  the set
of epimorphisms from $P_n$ to $P_k$, namely, 
\begin{align}\label{eq:a}
\Epi(P_n, P_k)=\{f \in \Hom(P_{n}, P_{k})\,:\, f([n])=[k]\}.
\end{align}
\begin{lem} \label{lem9} 
For $1\leq k \leq n-1$,
\begin{align}
l_{k}(n)=|\Epi(P_{n},P_{n-k+1})|/2.
\end{align}
\end{lem}
\begin{proof} 
Let $r=n-k+1$.  Denote by $\End_r(P_n)$  the subset of endomorphisms in $\End(P_n)$ such that $|f([n])|=r$ and $\L_k(n)$ the set of partitions induced
by endomorphisms in $\End_r(P_n)$.  By definition (see \eqref{eq:b}),   the integer $l_k(n)$ is the cardinality of $\L_k(n)$.

For each $f\in \End_r(P_n)$,  if 
 $f([n])=\{a, a+1, \ldots, a+r-1\}$ for some integer $a\in [n-r+1]$, 
 we define $\bar f\in \Epi(P_n, P_r)$ by $\bar f(x)=f(x)-a+1$. Then $f$ and $\bar f$ induce the same partition in $\L_k(n)$.
  Hence, we can consider $\L_k(n)$  as the set of partitions induced
by epimorphisms in 
 $\Epi(P_n, P_r)$.  

If  $\{A_1, \ldots, A_{r}\}$ is  a partition of $[n]$ 
induced by an 
 $f\in \Epi(P_n, P_r)$, then,  we can assume that 
$\min(A_1)\leq  \min(A_2)\leq \ldots \leq \min(A_{r})$. 
 Hence, we can
 identify $f$ 
with a permutation $\sigma\in \S_r$  by
$f(A_{\sigma(i)})=i$ for $i\in [r]$.  
Moreover, two blocks $A_i$ and $A_j$ are  adjacent in the arrangement 
$A_{\sigma(1)}\ldots A_{\sigma(r)}$ if and only if 
there are   two 
 consecutive integers $\alpha$ and $\beta$ such that $\alpha\in A_i$ and $\beta\in A_j$. 
We  show that there are exactly  two such permutations for a given  induced partition. 

Starting from a  partition $\{A_1, \ldots, A_{r}\}$ of $[n]$ 
induced by an $f\in \Epi(P_n, P_r)$, 
we  arrange  step by step the blocks $A_1, \ldots, A_i$ for $2\le i\le r$ such that $A_i$ is adjacent 
to the block $A_j$ containing $\min(A_i)-1$ and $j<i$.
Since  $\min(A_1)=1$ and $\min(A_2)=2$, there are two ways to 
arrange $A_1$ and $A_2$: $A_1A_2$ or $A_2A_1$.  
 Suppose that  the first $i$ ($\geq 2$) blocks have been arranged as 
$W_i:=A_{\sigma_i(1)}\ldots A_{\sigma_i(i)}$ with $\sigma_i\in \S_i$, then 
 $\min(A_{i+1})-1$ must belong to $A_{\sigma_i(1)}$ or $A_{\sigma_i(i)}$ because 
 any  two adjacent blocks in $W_i$ should stay adjacent in 
 all the $W_j$ for $i\leq j\leq r$.
 Hence there is only one way to insert $A_{i+1}$ in $W_i$: 
 at the left of  $W_i$ (resp.  right of  $W_i$)
 if $\min(A_{i+1})-1\in A_{\sigma_i(1)}$ (resp.  $A_{\sigma_i(i)}$) for $i\geq 2$.
  As there are two possibilities for $i=2$ we have thus proved that there are exactly  two 
 corresponding epimorphisms in $\Epi(P_n, P_r)$ 
  for a given  induced partition with $r$ blocks. 
 For example, starting from the induced   partition $\{\{1, 3, 5, 9\}, \{2, 4, 10\}, \{6, 8\}, \{7\}, \{11\}\}$  of $V(P_{11})$, 
 we obtain  the two corresponding arrangements:
  $$
  \{7\}\{6, 8\}\{1, 3, 5, 9\}\{2, 4, 10\}\{11\}\quad  \textrm{and} \quad
  \{11\}\{2, 4, 10\}\{1, 3, 5, 9\}\{6, 8\}\{7\}.
  $$
  This is  the desired result.
\end{proof}

\begin{lem}\label{lem7}  For $1\leq k \leq n$, 
$$
l_{k}(n)
=\frac{1}{2}|\Hom(P_{n},P_{n-k+1})|-|\Hom(P_{n},P_{n-k})|+\frac{1}{2}|\Hom(P_{n},P_{n-k-1})|.
$$
\end{lem}
\begin{proof}

By definition we have 
$\Hom(P_{n},P_{k})\setminus \Epi(P_{n},P_{k})=A\cup B$,
where 
\begin{align*}
A&= \{f\in \Hom(P_{n},P_{k})\,:\, f([n])\subseteq [k-1]\},\\
B&=\{f\in \Hom(P_{n},P_{k})\, :\, f([n])\subseteq [k]\setminus[1]\}.
\end{align*}
Hence
\begin{align}\label{le:hi}
| \Hom(P_{n},P_{k})|-| \Epi(P_{n},P_{k})|=|A|
+| B|-|A\cap B|.
\end{align}
Since 
$|A|=| B|=|\Hom(P_{n},P_{k-1})|$, and 
$$
|A\cap B|= |\{f\in \Hom(P_{n},P_{k})\,:\, f([n])\subseteq [k-1]\setminus [1]|
=|\Hom(P_{n},P_{k-2})|,
$$  
we derive from  \eqref{le:hi}   
 that
\begin{align*}
|\Epi(P_n, P_k)|=|\Hom(P_{n},P_{k})|-2|\Hom(P_{n},P_{k-1})|+|\Hom(P_{n},P_{k-2})|.
\end{align*}
The result follows then by  applying Lemma~\ref{lem9}.
\end{proof}

It follows from Lemma~\ref{lem7} and Theorem~2  that
\begin{align} \label{eq:le}
l_{k}(n)=\sum\limits_{i=0}^{n-2}2^{n-i-2}
\sum_{j\in \mathbb{Z}}(-A_{i,\,j}+2B_{i,\,j}-C_{i,\,j}),
\end{align} 
where 
\begin{align*}
 A_{i,\,j}=A^+_{i,\,j}-A^-_{i,\,j},\quad 
B_{i,\,j}=B^+_{i,\,j}-B^-_{i,\,j},\quad
C_{i,\,j}=C^+_{i,\,j}-C^-_{i,\,j},
\end{align*}
with 
\begin{align*}
A^+_{i,\,j}&={i \choose \lceil\frac{i}{2}\rceil-j(n-k+2)},\quad 
A^-_{i,\,j}={i \choose
\lfloor\frac{i+n-k}{2}\rfloor+1-j(n-k+2)},\\
B^+_{i,\,j}&={i \choose \lceil\frac{i}{2}\rceil-j(n-k+1)},\quad 
B^-_{i,\,j}={i \choose
\lfloor\frac{i+n-k-1}{2}\rfloor+1-j(n-k+1)},\\
C^+_{i,\,j}&={i \choose \lceil\frac{i}{2}\rceil-j(n-k)},\quad
C^-_{i,\,j}={i \choose
\lfloor\frac{i+n-k-2}{2}\rfloor+1-j(n-k)}.
\end{align*}
\begin{lem} \label{lem8} For $n\geq 2k$, 
\begin{align}\label{eq:le8}
\sum_{j\in \mathbb{Z}}(-A_{i,\,j}+2B_{i,\,j}-C_{i,\,j})=&
\left\{\binom{i+1}{\left\lfloor\frac{i+n-k}2\right\rfloor+1}+\binom{i+1}{\left\lfloor \frac{i+n-k}2\right\rfloor-n+k}\right\}\nonumber\\
&-2
\left\{\binom{i}{\left\lfloor\frac{i+n-k-1}2\right\rfloor+1}+\binom{i}{\left\lfloor \frac{i+n-k-1}2\right\rfloor-n+k}\right\}.
\end{align}
\end{lem}
\begin{proof}
Since $0\leq k\leq\frac{n}2$,
we have  $\frac{n}2\leq n-k\leq n-1$. Therefore,
\begin{enumerate}
\item if $j<0$,
then $\lceil \frac{i}2\rceil-j(n-k)\geq \lceil \frac{i}2\rceil+n-k\geq\lceil \frac{i}2\rceil+\frac{n}2\geq \lceil \frac{i}2\rceil+\frac{i}2+1\geq i+1$ because $i\leq n-2$. Similarly we have  $\lfloor\frac{i+n-k-2}{2}\rfloor+1-j(n-k)\geq i+1$. 
Hence, all the summands  $A_{i,\,j}$, $B_{i,\,j}$ and $C_{i,\,j}$  vanish;
\item if $j>0$,
then $\lceil \frac{i}2\rceil-j(n-k)\leq \lceil \frac{i}2\rceil-(n-k)\leq\lceil \frac{i}2\rceil-\lceil\frac{n}2\rceil\leq -1$ because $i\leq n-2$.
Hence, all $A^+_{i,\,j}$, $B^+_{i,\,j}$ and $C^+_{i,\,j}$  vanish;
\item if $j\geq 2$, then $\lfloor\frac{i+n-k}{2}\rfloor+1-j(n-k+2)\leq \lfloor\frac{n-2+n-k}{2}\rfloor+1-2(n-k+2)\leq\frac{3}2k-n-5\leq -1$. Similarly we have $\lfloor\frac{i+n-k-1}{2}\rfloor+1-j(n-k+1)\leq -1$ and $\lfloor\frac{i+n-k-2}{2}\rfloor+1-j(n-k)\leq -1$, so all $A^-_{i,\,j}$, $B^-_{i,\,j}$ and $C^-_{i,\,j}$  vanish.
\end{enumerate}
It follows that  the summation over $j\in \Z$ in \eqref{eq:le8} reduces to
$$
-A^-_{i,\,0}+2B^-_{i,\,0}-C^-_{i,\,0}-A^-_{i,\,1}+2B^-_{i,\,1}-C^-_{i,\,1}.
$$
Using ${n \choose k}+{n \choose k-1}={n+1 \choose k}$ to combine 
$A^-_{i,\,0}$ with $C^-_{i,\,0}$ and $A^-_{i,\,1}$ with $C^-_{i,\,1}$, respectively, 
we derive the desired formula.
\end{proof}

Now, we are in position to prove Theorem~1.
When $n\geq 2k$, by Lemma~\ref{lem8},   the summands in \eqref{eq:le} can be  written as
$$
2^{n-i-2}\sum_{j\in \mathbb{Z}}(-A_{i,\,j}+2B_{i,\,j}-C_{i,\,j})=D_{i+1} -D_{i},
$$
where 
$$
D_{i}=2^{n-i-1}
\left\{\binom{i}{\left\lfloor\frac{i+n-k-1}2\right\rfloor+1}+\binom{i}{\left\lfloor \frac{i+n-k-1}2\right\rfloor-n+k}\right\}.
$$
Substituting this into \eqref{eq:le} we obtain
$$
l_{k}(n)=\sum_{i=0}^{n-2}(D_{i+1} -D_{i})=D_{n-1},
$$
which is clearly equivalent to \eqref{eq:main}.

\subsection*{Acknowledgement}
We thank the referees for helpful suggestions on the initial version of this paper.
This work was  partially supported by 
 the grant ANR-08-BLAN-0243-03.

\end{document}